\documentclass[12pt]{article}
\usepackage{latexsym, amssymb}
\usepackage{amsmath,amsfonts,amsthm,latexsym, amssymb, wasysym, dsfont}
\usepackage{graphicx}

\newtheorem{theorem}{Theorem}

\newtheorem{proposition}[theorem]{Proposition}
\newtheorem{conjecture}[theorem]{Conjecture}
\newtheorem{definition}[theorem]{Definition}

\newtheorem{lemma}[theorem]{Lemma}
\begin{document}

\title{Number Theories}
\author{Patrick St-Amant}
\date{}

\maketitle

\begin{abstract}We will see that key concepts of number theory can be
defined for arbitrary operations.
We give a generalized distributivity for hyperoperations (usual arithmetic operations and operations going beyond exponentiation) and a generalization of the fundamental theorem of arithmetic for hyperoperations. We also give a generalized definition of the prime numbers that are associated to an arbitrary n-ary operation and take a few steps toward the development of its modulo arithmetic by  investigating a generalized form of Fermat's little theorem. Those constructions give an interesting way to interpret diophantine equations and we will see that the uniqueness of factorization under an arbitrary operation can be linked with the Riemann zeta function. This language of generalized primes and composites can be used to restate and extend certain problems such as the Goldbach conjecture.
\end{abstract}

\section{Introduction}
Initially, the goal of the present study was to find interesting properties involving hyperoperations beyond exponentiation. One obstacle is that beyond exponentiation it becomes very difficult to have any examples involving actual integers. However, a generalization of the fundamental theorem of arithmetics has been isolated with the objective that eventually it will be possible to find a `fundamental theorem of hyperarithmetic'. This result required a more general definition of primes and this prompted many connections and avenues of investigations. It is our hope that enough steps have been taken to point out the vastness of this unknown arithmetic world and enough results have been demonstrated to indicate its approachability.

Many systems for hyperoperations have been introduced for different purposes in the literature, in particular \cite{Ackermann}, \cite{Knuth} and \cite{Conway}. Hyperoperations give rise to immense numbers which, according to Littlewood \cite{Littlewood}, fascinated investigators such as Archimedes himself. We can now construct different notations to contemplate and manipulate incredibly large numbers. In particular, we can write the Graham's number \cite{Graham} by using the Knuth arrow notation \cite{Knuth}. Another interesting application in Computer Science is the use systems of hyperoperations to handle floating point overflow \cite{Clenshaw}. Strangely, the investigation of hyperoperations has not yet become mainstream although most of number theory and Abstract Algebra rely solely on the common operations of addition and multiplication. By extending the study of number theory to arbitrary operations, it becomes possible to consider a whole  range of questions, connections and applications.

The implicit goal of this article is to go back at the initial concepts of arithmetic operations and extend them to give a richer theory which present an opportunity to cast a new light on classical concepts by studying them in a wider context. The explicit objective of this article is threefold. First, we will show that the common arithmetic operations can be extended further by the use of hyperoperations. Usually, when generalizing concepts we lose important properties. One convenient way to keep those properties is to generalize the properties. We will see that the concepts of distributivity and the
fundamental theorem of arithmetics, which can be seen as the property of uniqueness of factorization, can also be demonstrated for hyperoperations beyond the common operations.

Secondly, from those investigations naturally arises the idea of primes associated to arbitrary binary operations over the positive integers. Classically, the extension of the concept of primes or related concept of irreducibility have been extensively studied in many different area of mathematics, in particular, prime ideals which usually rely on operations involve in commutative rings is to topic of a large range of investigations \cite{Kunz} \cite{Fulton}.  We will see that this concept of generalized primes, which is dependant on the operations themselves, invites interesting connections with Diophantine equations and Dirichlet series. We will see that such a connection opens a door to use tools related purely to the study of primes for questions about Diophantine equations or Dirichlet series and vice versa.

Thirdly, in section \ref{generalsection} we will extend the definition of primes associated to binary operations by defining primes associated to arbitrary partial functions of $n$ variables over arbitrary sets. . All those generalizations can be seen as introducing a language of generalized primes and composites and invites us to rewrite and suggest diverse problems by using this language. A suggested conjecture reminds us strongly of the Goldbach conjecture. Briefly, we will see that these investigations prompt further generalizations and suggest a general context for investigations which can be categorical.

The title was chosen to point out that number theory can be done based on arbitrary operations which can differ from the common operations of addition and multiplication or from another perspective, that there is a multitude of number theories. During the present study, many doors are opened and it is the impression of the author that beyond those doors there is a exotic and wide world involving giant finite numbers and fascinating entities such as the generalized primes.

\section{Hyperoperations}

There are many ways to define what lies beyond exponentiations and
each possibility gives rise to very different systems. For clarity and
completeness we will present explanations and an independent definition of
hyperoperations.

\subsection{Preliminaries}

Let us take an arbitrary integer and denote it by $x$. In the following, we will restrict ourselves to integers, but this does not exclude definitions which would consider other types of numbers. We can
perform the addition of many objects $x$ using the symbol $+$ and
write

$$x+x+x+...+x$$

Multiplication is defined in the following manner and can be seen as a prefix notation since the number of occurrences of the $x$'s appear on the left.
$$x+x=2\cdot x$$
$$x+x+x=3\cdot x$$
$$x+x+x+x=3\cdot x$$
$$...$$

Similarly, the exponential notation is defined in the following manner and can be seen a suffix notation.

$$x\cdot x=x^2$$
$$x\cdot x\cdot x=x^3$$
$$x\cdot x\cdot x\cdot x=x^4$$
$$...$$

Instead of using the power notation, we can define exponentiation in the same format as addition and multiplication. We will take
the operation $\oplus_2$ to indicate exponentiation. Note that here we will use a suffix notation. We write:

$$x\cdot x=2\oplus_2 x$$
$$x\cdot x\cdot x=3\oplus_2 x$$
$$x\cdot x\cdot x\cdot x=4\oplus_2 x$$
$$...$$

We could go a step further by defining binary operations beyond exponentiation in the following
manner:

$$x\oplus_2 x=2\oplus_3 x$$
$$x\oplus_2 x\oplus_2 x=3 \oplus_3x$$
$$x\oplus_2 x\oplus_2 x\oplus_2 x=4 \oplus_3 x$$
$$...$$

But, we see that at this point the parenthesis becomes very important since we lose the associativity property. For example $(x\oplus_2 x)\oplus_2 x\neq x\oplus_2
(x\oplus_2 x)$ for $x>2$. Thus, in the formal definition, we will keep track of the parenthesis by using superscript notation on the operations.

It is possible to continue this pattern and define an infinite number of
new operations. After giving the formal definitions, we will investigate
some of the properties of those new binary operations.

This extended notation provides a way to investigate the
property of larger numbers. For example $(x\oplus_2 x)\oplus_2 x=x^{(x^x)}=x^{x^x}$ becomes very
large just by taking $x=10$ and  $((x\oplus_2 x)\oplus_2 x)\oplus_2 x=x^{x^{x^x}}$ becomes very large
only by taking $x=3$.

\subsection{A Definition of Hyperoperations}\label{hyperarithmetic}

For clarity, before giving a complete definition of hyperoperations, we formally define the multiplication ($\oplus_1$) and the exponentiation ($\oplus_2$) notation. Note that we will take $\mathbb{N}$ to be the positive integers positive integers $\{1,2,3,...\}$.

\begin{definition}\label{operations}For $i\in\{1,2\}$,
$n\in\mathbb{N}$, an
object $x$ and assuming
the knowledge of the usual addition and its properties, we have
that
\begin{list}{}{}
\item i) $\oplus_0$ is the usual addition operation $+$

\item ii) $x=(1)\oplus_1 x$ and $x=(1)\oplus_2 x$

\item iii) $(n)\oplus_{i+1}x\oplus_i x=(n+1)\oplus_{i+1}x$.

\item iv) $x\oplus_i(n)\oplus_{i+1}x =(n+1)\oplus_{i+1}x$.

\end{list}
\end{definition}

From the above definition arises the notation described informally
in the preceding section for addition, multiplication and
exponentiation. For example, using the identities of the
definition, we find that $x+x+x$ can be written as $3\oplus_1 x$
or as in the usual notation $3\cdot x$.

\bigskip

\begin{tabular}{lll}
$x+x+x$&=&$(1)\oplus_1 x+x+x$\\
&=&$(1+1)\oplus_1 x+x$\\
&=&$(2)\oplus_1 x+x$\\
&=&$(2+1)\oplus_1 x$\\
&=&$3\oplus_1 x$\\
&=&$3\cdot x$\\
\end{tabular}

\bigskip

As noted before, already at the exponential operation we lose
associativity. Hence, if we want to define operations beyond
the exponential we will need to not assume associativity in the definition of hyperoperations.

When we collect the $x$'s we need to devise a way for the higher operation to keep track of the parentheses. This will be done by using a superscript notation on the
higher indexed operation as follows.

First, we devise a way to write the parentheses of a formula by using a superscript notation on the operations.

\begin{definition}We define the set of \emph{hyperoperations-formulas} $H$ as:
\begin{itemize}
\item[1.] if $n\in\mathbb{N}$ then $n\in H$,

\item[2.] if $x,y\in H$ and $i\in\mathbb{N}$ then $(x\oplus_i y)\in H$,
\item[3.] if $x,y\in H$ and $i,h\in\mathbb{N}$ then $x\oplus^h_i y\in H$.
\end{itemize}
\end{definition}

\begin{definition} For $x,y\in H$ and for any index name $j,k$, we have
\begin{itemize}
\item[1.] if $x$ and $y$ are in $\mathbb{N}$, then $$(x\oplus_j y)=x\oplus_j^0 y$$

\item[2.] Let $m=max(n_x,n_y)$ such that $n_x$ and $n_y$ are the maximum integers appearing on the superscript of the operations of $x$ and $y$, respectively. If there are no parentheses in $x$ and $y$, then $$(x\oplus_k y)=x\oplus^{m+1}_k y$$

\end{itemize}
\end{definition}

The superscript indicates in what order the operation is performed
and can be used instead of parenthesis. For example, assuming that $x,y,z\in \mathbb{N}$, we have

\bigskip
\begin{center}\begin{tabular}{lll}
$(((x\oplus x)\oplus x)\oplus x)$&=&$((x\oplus^0 x\oplus x)\oplus x)$\\
&=&$(x\oplus^0 x\oplus^1 x\oplus x)$\\
&=&$x\oplus^0 x\oplus^1 x\oplus^2 x$\\
\end{tabular}\end{center}

\bigskip

and

\bigskip

\begin{center}\begin{tabular}{lll}
$((x\oplus (x\oplus x)\oplus y\oplus^4 z)$&=&$((x\oplus x\oplus^0 x)\oplus y\oplus^4 z)$\\
&=&$(x\oplus^1 x\oplus^0 x\oplus y\oplus^4 z)$\\
&=&$x\oplus^1 x\oplus^0 x\oplus^5 y\oplus^4 z$\\
\end{tabular}\end{center}

\bigskip

In the two examples, all the operations are the same, but this need not be always like this. When they are all the same we can write them using a higher operation. We will formalize the following notation in definition \ref{operations}, so that writing a higher operation keeps track of the parentheses.

$$4\oplus_{k+1}^{[0,1,2]}x=x\oplus^0 x\oplus^1 x\oplus^2 x=(((x\oplus x)\oplus x)\oplus x)$$
and
$$4\oplus_{k+1}^{[1,0,2,0]}x=x\oplus^1 x\oplus^0 x\oplus^2 x\oplus^0 x=((x\oplus (x\oplus x))\oplus (x\oplus x))$$

The formal definition of hyperoperations is given below and follows the format of our definition of multiplication and exponentiation and includes the superscript notation of the parentheses.

\begin{definition}\label{operations}For all $i,j\in\mathbb{N}$, an
object $x$ and assuming
the knowledge of the usual addition and its properties, we have

\begin{list}{}{}
\item i) $\oplus_0$ is the usual addition operation $+$

\item ii) $x=(1)\oplus^{[\,]}_j x$ with $[\,]$ an empty sequence.

\item iii) $(n)\oplus_{i+1}^{[S]}x\oplus^{[a]}_i
x=(n+1)\oplus^{[S,a]}_{i+1}x$ with $n,a\in\mathbb{N}$ and $[S]$ is
a sequence.

\item iv) $x\oplus^{[b]}_i(n)\oplus_{i+1}^{[S]}x
=(n+1)\oplus^{[b,S]}_{i+1}x$ with $n,b\in\mathbb{N}$ and $[S]$ is
a sequence.

\end{list}
\end{definition}

Using this notation with superscripts we revisit a previous example which can be
written as the following. Note that since addition is associative,
we can write $x+x+x$ as $x+^{[0]}x+^{[0]}x$.

\bigskip

\begin{tabular}{lll}
$x+^{[0]}x+^{[0]}x$&=&$(1)\oplus_1^{[\,]} x+^{[0]}x+^{[0]}x$\\
&=&$(1+1)\oplus_1^{[0,0]} x+x$\\
&=&$(2)\oplus_1^{[0,0]} x+x$\\
&=&$(1+2)\oplus_1^{[0,0,0]} x$\\
&=&$3\oplus_1^{[0,0,0]} x$\\
&=&$3\cdot^{[0,0,0]} x$\\
&=&$3\cdot x$\\
\end{tabular}

\bigskip

Here are two more elaborate examples:

\bigskip

\begin{tabular}{lll}
$x^{x^{x^x}}$&=&$(((x\oplus_2 x)\oplus_2 x)\oplus_2 x)$\\
&=&$x\oplus^0 x\oplus^1 x\oplus^2 x$\\
&=&$(1)\oplus^{[\,]}_3 x\oplus^0 x\oplus^1 x\oplus^2 x$\\
&=&$(2)\oplus^{[0]}_3 x\oplus^1 x\oplus^2 x$\\
&=&$(3)\oplus^{[0,1]}_3 x\oplus^2 x$\\
&=&$(4)\oplus^{[0,1,2]}_3 x$\\
\end{tabular}
\bigskip

\bigskip

\begin{tabular}{lll}
$(x^x)^{x^{x^x}}$&=&$(((x\oplus_2 x)\oplus_2 x)\oplus_2 (x\oplus_2 x))$\\
&=&$x\oplus^0_2 x\oplus^1_2 x\oplus^2_2 x\oplus^0_2 x$\\
&=&$x\oplus^0_2 x\oplus^1_2 (1)\oplus^{[\,]}_3 x\oplus^2_2 x\oplus^0_2 x$\\
&=&$x\oplus^0_2 (2)\oplus^{[1]}_3 x\oplus^2_2 x\oplus^0_2 x$\\
&=&$x\oplus^0_2 (3)\oplus^{[1,2]}_3 x\oplus^0_2 x$\\
&=&$x\oplus^0_2 (4)\oplus^{[1,2,0]}_3 x$\\
&=&$(5)\oplus^{[0,1,2,0]}_3 x$\\
\end{tabular}

\bigskip

Sometimes we will write $\rightarrow$ instead of $[0,1,2,3,...]$ which indicates that the operations are composed from the left to the right and similarly we will write $\leftarrow$ for $[...,3,2,1,0]$. For exponentiation and for hyperoperations beyond exponentiation we do not have associativity. For example, $(4)\oplus^{\rightarrow}_3 x\neq (4)\oplus^{\leftarrow}_3 x$, since
$$(4)\oplus^{\rightarrow}_3 x=(((x\oplus_2 x)\oplus_2 x)\oplus_2 x)=x^{x^{x^x}}$$
$$(4)\oplus^{\leftarrow}_3 x=(x\oplus_2 (x\oplus_2 (x\oplus_2 x)))=((x^x)^x)^x=x^{x^3}$$

Note that, in general, we have the following expressions:
$$(n)\oplus^{\rightarrow}_3 x=(((x\oplus_2 x)\oplus_2 ...)\oplus_2 x)=\underbrace{x^{x^{\, .^{\, .^{\, .^{x}}}}}}_n$$

$$(n)\oplus^{\leftarrow}_3 x=(x\oplus_2 (...\oplus_2(x\oplus_2 (x\oplus_2 x)))=x^{x^{n-1}}$$

The use of the superscript notation becomes especially important when we want to investigate properties. In particular, we will see that sequences of even and odd integers appear in a generalization of distributivity.

\subsection{Zero}

We now look at the object zero with this extended
notation for $x\in\mathbb{N}$. For addition, we know that $0+x=x$
and $x+0=x$. For multiplication we have that $0\cdot x= 0$ and
$x\cdot 0=0$. For exponentiation, we lose the property commutativity since
$x\oplus_2 0=0^x=0$ and $0\oplus x=x^0=1$ for $x\neq 0$.

For the operation $\oplus_3$ it becomes a little more subtle. If
the superscript arrow is as defined in the previous section, then we have $4\oplus_3^{\rightarrow}x=x^{x^{x^x}}$. We can
conceptualize that $0\oplus_3^{\rightarrow}x=0$ for all $x$, but
what can we say about $x\oplus_3^{\rightarrow}0$?

In some contexts, such as Calculus, $0^0$ is taken to be an indeterminate, but in
other contexts it is taken to be equal to $1$. If we assume that
$0^0=1$ then we find the following equations.

$$1\oplus_3^{\rightarrow}0=0=0$$
$$2\oplus_3^{\rightarrow}0=0^{0}=1$$
$$3\oplus_3^{\rightarrow}0=0^{0^0}=0$$
$$4\oplus_3^{\rightarrow}0=0^{0^{0^0}}=1$$
$$...$$

Which can be summarized in the following proposition.

\begin{proposition}Let $n\in\mathbb{N}$, if we say that $0^0=1$, then
$$n\oplus_3^{\rightarrow}0=\begin{cases}0 \mbox{ if n is odd} \\
1 \mbox{ if n is even}\end{cases}$$\end{proposition}
\begin{proof}By induction on even integers and by induction on odd
integers.\end{proof}

\subsection{Higher Distributivity}

In arithmetic, distributivity is one of the fundamental identities
regarding the interaction between two operations. In this section we
will study identities that combine two operations such that there
is a difference of $1$ between their subscripts and that the leftmost
operation has a greater subscript than the other operation. In a few words,
distributivity informs us that many occurrences of a term can be
grouped together. For example, $n(x+y)=nx+ny$ tells us that we can
group the $x$ and $y$ terms together. The following proposition
generalizes this up to a pair of operations $\oplus_3$ and
$\oplus_2$. Note that $\mathbb{N}^+$ will denote the non-negative integers.

\begin{theorem}For all $i\in{1,2,3}$, we have
$$(n)\oplus_i^{(1,2,...,n-1)}(x\oplus_{i-1}y)
=((n)\oplus_{i}^{Even}x\oplus_{i-1}(n)\oplus^{Odd}_{i}y)$$
with $n,x,y\in\mathbb{N}$ and for $$Even=[2n-2, 2n-4,...,4,2]$$ and $$Odd=[1,3,...,2n-5,2n-3].$$\end{theorem}

Before giving the proof, we will verify the theorem for a specific
example by making a free use of the notation. We have $$(3)\oplus^{[0,1]}_3(x\oplus_2y)=(3)\oplus_3^{[4,2]}x\oplus_2^{[0]}(3)\oplus_3^{[1,3]}y$$

\noindent Since
\bigskip

\begin{tabular}{lll}
$(3)\oplus^{[0,1]}_3(x\oplus_2y)$&=&$3\oplus^{[0,1]}_3(y^x)$\\
&=&$y^x\oplus_2^{[0]} y^x\oplus_2^{[1]} y^x$\\
&=&$(y^x\oplus_2 y^x)\oplus_2 y^x$\\
&=&$((y^x)^{y^x})\oplus_2 y^x$\\
&=&$(y^{xy^x})\oplus_2 y^x$\\
&=&$(y^x)^{y^{xy^x}}$\\
&=&$y^{xy^{xy^x}}$\\
\end{tabular}

\bigskip

\begin{tabular}{lll}
$(3)\oplus_3^{[4,2]}x\oplus_2^{[0]}(3)\oplus_3^{[1,3]}y$&=&$x\oplus_2^{(4)}
x\oplus^{[2]}_2 x\oplus_2^{[0]}{(3)}\oplus_3^{[1,3]}y$\\
&=&$x\oplus_2^{[4]} x\oplus^{[2]}_2
x\oplus_2^{[0]}y\oplus_2^{[1]}y\oplus_2^{[3]}y$\\
&=&$x\oplus_2 ((x\oplus_2 ((x\oplus_2y)\oplus_2y))\oplus_2y)$\\
&=&$x\oplus_2 ((x\oplus_2 ((y^x)\oplus_2y))\oplus_2y)$\\
&=&$x\oplus_2 ((x\oplus_2 (y^{(y^x)}))\oplus_2y)$\\
&=&$x\oplus_2 ((y^{x(y^x)})\oplus_2y)$\\
&=&$x\oplus_2 (y^{(y^{x(y^x)})})$\\
&=&$y^{x(y^{x(y^x)})}$\\
&=&$y^{xy^{xy^x}}$\\
\end{tabular}

\bigskip

\begin{proof}
Since we know that $n(x+y)=nx+ny$ and because addition is
associative, the identity is also true for $i=1$

By a well known property of exponents, we have $(xy)^n=x^ny^n$. This can be written as $$n\oplus_2 (x\oplus_1 y)=(n\oplus_2 x) \oplus_1 (n\oplus_2 y).$$ Since multiplication is
associative, we find that the statement of the theorem is also true for $i=2$.

We will prove by induction on $n$ that the distributivity identity
is true for $i=3$. We verify that it is true for $n=0$ and $n=1$. Suppose that it is true
for $n=k$ then we will show that it is true for $n=k+1$.

Let $A=(k+1)\oplus_3^{[1,2,...,k]}(x\oplus_{2}y)$. By definition \ref{operations} we can write

$$A=(k)\oplus_3^{[1,2,...,k-1]}(x\oplus_{2}y)\oplus_2^k (x\oplus_2y)$$
\noindent By definition of the superscript, we can write the superscript $k$ on the right side of the previous equation as a parenthesis.Thus we have $$A=((k)\oplus_3^{[1,2,...,k-1]}(x\oplus_{2}y)\oplus_2 (x\oplus_2y)).$$

Another well known property of exponents is $(y^{x})^a=(y^a)^x$.
In the hyperoperation notation we have the identity
$$a\oplus_2(x\oplus_2 y)=x\oplus_2(a\oplus_2 y)$$

Using this property of exponents for
$a=(k)\oplus_3^{[1,2,...,k-1]}(x\oplus_{2}y)$ we find

$$A=((k)\oplus_3^{[1,2,...,k-1]}(x\oplus_{2}y)\oplus_2(x\oplus_2y))
=(x\oplus_2((k)\oplus_3^{[1,2,...,k-1]}(x\oplus_{2}y)\oplus_2y))$$

Since we assume that the identity is true for $n=k$ we have that $$(k)\oplus_3^{[1,2,...,k-1]}(x\oplus_{2}y)=(k)\oplus_{3}^{Even'}x\oplus^{[0]}_{2}(k)\oplus^{Odd'}_{3}y$$
Where  $$Even'=[2k-2, 2k-4,...,4,2]$$ and $$Odd'=[1,3,...,2k-5,2k-3].$$
Hence, we can write

$$A=(x\oplus_2((k)\oplus_3^{[1,2,...,k-1]}(x\oplus_{2}y)\oplus_2y))
=(x\oplus_2((k)\oplus_{3}^{Even'}x\oplus^{[0]}_{2}(k)\oplus^{Odd'}_{3}y\oplus_2y))$$
By definition \ref{operations}, we can write $(k)\oplus_{3}^{Even'}x$ and $(k)\oplus^{Odd'}_{3}y$ in terms of the operation $\oplus_2$ only. Since the maximum integer of the $Even'$ and $Odd'$ sequences is $2k-2$, we find, by using definition \ref{operations} again, that

$$A=(x\oplus_2((k)\oplus_{3}^{Even'}x\oplus^{[0]}_{2}(k)\oplus^{Odd'}_{3}y\oplus_2y))
=x\oplus^{[2n]}_2(k)\oplus_{3}^{Even}x\oplus^{[0]}_{2}(k)\oplus^{Odd}_{3}y\oplus^{[2n-1]}_2y$$
Thus, we find
$$A=((k+1)\oplus_{3}^{Even}x\oplus_{2}(k+1)\oplus^{Odd}_{3}y),$$
for  $$Even=[2k,2k-2, 2k-4,...,4,2]$$ and $$Odd=[1,3,...,2k-5,2k-3,2k-1].$$
Hence, by induction, the distributive identity is true for
$i=3$.

\end{proof}

Note that the above theorem is true for
$n\in\mathbb{N}$. Would it still be true if we replace $\mathbb{N}$
by the real or complex numbers? Also, other relations could be considered. In particular, relations where the difference between the subscripts is $k\in\mathbb{N}$.

\subsection{Higher Fundamental Theorem of Arithmetic}

The fundamental theorem of arithmetic tells us that there is a unique way to write
a positive integer $n>1$ in the form

$$p_{a_1}^{b_1}\cdot p_{a_2}^{b_2}\cdot p_{a_3}^{b_3}\cdot ...\cdot
p_{a_h}^{b_h}$$

\noindent such that $p_{a1},p_{a2},...,p_{am}$ are
primes, $p_{a1}<p_{a2}<...<p_{am}$ and
$b_1,b_2,...,b_m$ are positive integers.

We remark that the formulation of the fundamental theorem of arithmetic depends on the operations of multiplication and exponentiation. In the following, we will see that a similar formulation exists for the operation of exponentiation and the operation $\oplus_3$.

The definition of primes relies on the binary operation of multiplication. Let's define the concept of primes for exponentiation.

\begin{definition}We will say that $q$ is an
exponential prime or a $\oplus_2$-prime if $q$ cannot be written as $u^v$ for some positive integer $u,v$ such that $u>1$ and $v>1$.\end{definition}

\begin{lemma}\label{lemmaunique}$q$ is an exponential prime if and only if the exponents $b_1,b_2,...,b_m$
of the unique prime factorization of $q\neq 1$ satisfy $\gcd
(b_1,b_2,...,b_m)=1$.\end{lemma}
\begin{proof}($\Rightarrow$) Suppose $\gcd (b_1,b_2,...,b_m)=k\neq 1$
then we can write $q=c^k$ for some positive integer $c$, hence a contradiction with the fact that $q$ is a $\oplus_2$-prime.

($\Leftarrow$)Suppose that $q\neq 1$ is not an exponential prime. Hence,
$q=u^v=p^{b_1}_{a_1}p^{b_2}_{a_2}...p^{b_m}_{a_m}$ for $v>1$ ,
since each $p_i$ are primes, we have that $p_i\mid u$, thus
$(p^{e_1}_{a_1} p^{e_2}_{a_2}...p^{e_t}_{a_t})^v=p^{b_1}_{a_1}
p^{b_2}_{a_2}...p^{b_t}_{a_t}$. By the Fundamental Theorem of
Arithmetic, for all $k$ we must have
$ve_k=b_k$, which contradicts the fact that $\gcd (b_1,b_2,...,b_m)=1$.
\end{proof}

We will generalize the fundamental theorem of arithmetic in the following theorem. It would be interesting to go beyond this and find a statement which could be seen as the `Fundamental Theorem of Hyperarithmetic'.  Note that the exponential primes is a synonym for the $\oplus_2$-primes and the usual primes (primes associated to multiplication) is a synonym for the $\oplus_1$-primes.

Due to the fact that ${b_1}\oplus_i^{\leftarrow}q_{a_1}=q_{a_1}^{q_{a_1}^{b_1-1}}$ becomes an immense number for small integers input, almost all examples are out of reach. Note that in the case where $b_1=1$, we find that $q_{a_1}^{q_{a_1}^{b_1-1}}=q_{a_1}$.

\begin{theorem}\label{uniquefactorization}For $i\in\{2,3\}$ and $b_j>0$, there is a unique way to write a
positive integer in the form $$[[[({b_m}\oplus_i^{\leftarrow}q_{a_m})\oplus_{i-1}...]\oplus_{i-1}
({b_3}\oplus_i^{\leftarrow}q_{a_3})]\oplus_{i-1}
({b_2}\oplus_i^{\leftarrow}q_{a_2})]\oplus_{i-1}
({b_1}\oplus_i^{\leftarrow}q_{a_1}),$$
where each $q_{a_k}$ is a $\oplus_i$-prime and $q_{a_k}$ does not divide $$[({b_m}\oplus_i^{\leftarrow}q_{a_m})\oplus_{i-1}...]\oplus_{i-1}
({b_{k+1}}\oplus_i^{\leftarrow}q_{a_{k+1}})$$ for all $k$.
\end{theorem}

Before giving the proof, to help intuition, we can write the factorization in the statement of the theorem for $i=3$ as
$$\underbrace{(q_{a_1}^{b_1-1})^{(q_{a_2}^{b_2-1})^{\,\,\, .^{\,\,\, .^{\,\,\, .^{\,\,\,(q_{a_m}^{b_m-1})}}}}}}_m$$

\begin{proof}

For $i=2$ we have
$$[[[({b_m}\oplus_2^{\leftarrow}q_{a_m})\oplus_{1}...]\oplus_{1}
({b_3}\oplus_2^{\leftarrow}q_{a_3})]\oplus_{1}
({b_2}\oplus_2^{\leftarrow}q_{a_2})]\oplus_{1}
({b_1}\oplus_2^{\leftarrow}q_{a_1}).$$
Since $\oplus_2$ is the usual exponentiation and $\oplus_1$ is the
usual multiplication we have and because multiplication is associative, the superscript `$\leftarrow$' and the box brackets can be ignored. Thus
$$(q_{a_m}^{b_m})\cdot...\cdot
(q_{a_3}^{b_3})\cdot
(q_{a_2}^{b_2})\cdot
(q_{a_1}^{b_1}),$$
where
 $(q_{a_m}^{b_m})\cdot ...\cdot (q_{a_{k+2}}^{b_{k+2}})\cdot (q_{a_{k+1}}^{b_{k+1}})$ cannot be divided by $q_k$. Hence, the theorem is true for $i=2$ by the fundamental theorem of arithmetic.

For $i=3$, we need to show that every positive integer $n$ can
written in the form seen in the statement of the theorem and we need to show that the representation in this form is unique.

We will show the first part by constructing for each $n$ such a
representation. By the fundamental theorem of arithmetic we have
$$n=p^{d_1}_{c_1} p^{d_2}_{c_2}...p^{d_t}_{c_t}$$
Let $\gcd (d_1,d_2,...,d_t)=g$ and
$gd_1'=d_1,gd_2'=d_2,...,gd_t'=d_t$, thus we have
$$n=(p^{d'_1}_{c_1} p^{d'_2}_{c_2}...p^{d'_t}_{c_t})^g.$$

Let $q_{a_1}=p^{d_1'}_{c_1} p^{d_2'}_{c_2}...p^{d_t'}_{c_t}$, since
$\gcd (d_1',d_2',...,d_t')=1$, we have by the previous lemma that
$q_{a_1}$ is an exponential prime.

Recall that $x\oplus_3^{\leftarrow}y$ can be written in usual notation
as $y^{y^{x-1}}$. Take the maximal integer $b_1\geq 0$ such that
$g=q_{a_1}^{b_1}n_1$ for some positive integer $n_1$. We now have
$$n=q_{a_1}^{(q_{a_1}^{b_1}n_1)}=(q_{a_1}^{q_{a_1}^{b_1}})^{n_1},$$
where $q_{a_1}\nmid n_1$ and we also have that  $n_1<n$. We repeat all the
previous steps starting with $n_1$ instead of $n$. This whole process will eventually stop since $0<...<n_2<n_1<n$.

We now show the uniqueness of representation. Suppose distinct representations such that
$(q_{a_1}^{q_{a_1}^{b_1-1}})^u=(q^{q_{c_1}^{d_1-1}}_{c_1})^v$ where $q_{a_1}$ and $q_{c_1}$ are exponential
primes and $$u=[[({b_m}\oplus_2^{\leftarrow}q_{a_m})\oplus_{1}...]\oplus_{1}({b_3}\oplus_2^{\leftarrow}q_{a_3})]\oplus_{1}({b_2}\oplus_2^{\leftarrow}q_{a_2})$$
and $$v=[[({d_{m'}}\oplus_2^{\leftarrow}p_{c_m})\oplus_{1}...]\oplus_{1}({d_3}\oplus_2^{\leftarrow}p_{c_3})]\oplus_{1}({d_2}\oplus_2^{\leftarrow}p_{c_2}).$$
Let $r=q_{a_1}^{b_1-1} u$ and $s=q_{c_1}^{d_1-1} v$ so that $q_{a_1}^r=q_{c_1}^s$.

Suppose $r\neq s$. Let $r=kr'$ and $s=ks'$ such that $\gcd (r',s')=1$ for $k$ a positive integer and $r,s$ are not both equal to $1$, thus we have $q_{a_1}^{r'}=q^{s'}_{c_1}$. Let $q_{a_1}=p_{e_1}^{f_1}p_{e_2}^{f_2}...p_{e_n}^{f_n}$ and
$q_{c_1}=p_{e_1}^{h_1}p_{e_2}^{h_2}...p_{e_n}^{h_n}$, this is because the $p$'s are primes which implies that $q_{a_1}$ and $q_{c_1}$ must have the same factors.
By the fundamental theorem of arithmetic we also must have that
$f_i r'=h_i s'$ for all $i$. Assuming that $r'\neq 1$ we can write this differently as $f_i=\frac{h_i s'}{r'}$.  Since $\gcd (r',s')=1$, we have that $r'\mid h_i$ for all i. This indicates that $\gcd (h_1,h_2,...,h_n)\neq 1$, a contradiction, by lemma \ref{lemmaunique}, with the assumption that $q_{c_1}$ is an exponential prime. Similarly, if $s'\neq 1$, we find a contradiction. Therefore we must have that $r'=s'=1$, which means that $r=k=s$ and we must conclude that $q_{a_1}=q_{c_1}$.

We now have
$$(q_{a_1}^{q_{a_1}^{b_1-1}})^u=(q^{q_{a_1}^{d_1-1}}_{a_1})^v$$ If $b_1=d_1$ we
are finished since we repeat the whole process starting with $u,v$ and this process must stop
since the exponents gets smaller at each step.

Applying the logarithm on each side of the previous equation
we find
$$q_{a_1}^{(b_1-1)u}=q_{a_1}^{(d_1-1)v}.$$
Without loss of generality, assume $b_1-1>d_1-1$. Define $z$ such that  $z=(b_1-1)-(d_1-1)>0$.
Hence
$$q_{a_1}^{(b_1-1)-(d_1-1)}u=v.$$
$$q_{a_1}^{z}u=v.$$
Thus $q_{a_1}$ must divide $v$, a contradiction with the assumption of the statement of the theorem. Therefore, we can conclude that we have uniqueness of representation.

\end{proof}

In the following section we will formally define what is a prime
associated to an operation and find that the only prime associated to addition is $1$. We could also include the case $i=1$ in theorem \ref{uniquefactorization}, by modifying the theorem in the following manner:

\begin{theorem}For $i\in\{1,2,3\}$ and $b_j>0$, there is a unique way to write a
positive integer in the form $$[[[({b_m}\oplus_i^{\leftarrow}q_{a_m})\oplus_{i-1}...]\oplus_{i-1}
({b_3}\oplus_i^{\leftarrow}q_{a_3})]\oplus_{i-1}
({b_2}\oplus_i^{\leftarrow}q_{a_2})]\oplus_{i-1}
({b_1}\oplus_i^{\leftarrow}q_{a_1}),$$
where each $q_{a_k}$ is a $\oplus_i$-prime and such that if $$[({b_m}\oplus_i^{\leftarrow}q_{a_m})\oplus_{i-1}...]\oplus_{i-1}
({b_{k+1}}\oplus_i^{\leftarrow}q_{a_{k+1}})\neq 0$$ then $q_{a_k}$ does not divide it.
\end{theorem}

\begin{proof}
Since for $i=1$ and $i=2$ it never occurs that $[({b_m}\oplus_i^{\leftarrow}q_{a_m})\oplus_{i-1}...]\oplus_{i-1}
({b_{k+1}}\oplus_i^{\leftarrow}q_{a_{k+1}})=0$, thus the theorem is true for $i=1$ and $i=2$ by theorem \ref{uniquefactorization}.

For $i=1$, we have that every positive integer $n$ can be written uniquely in the form $n=n\oplus_1 1=n\cdot 1$ and since $1$ divides every positive integer, the only possible representation is when  $$[({b_m}\oplus_i^{\leftarrow}q_{a_m})\oplus_{i-1}...]\oplus_{i-1}
({b_{k+1}}\oplus_i^{\leftarrow}q_{a_{k+1}})= 0.$$
Thus, the unique representation is $n=0+n\oplus_1 1=0+n\cdot 1$.
\end{proof}

The form of representation given in theorem \ref{uniquefactorization} is not necessarily the only form which gives rise to a uniqueness of representation. One other possibility could be

$$({c_1}\oplus_i^{\rightarrow}q_{a_1})\oplus_{i-1}
[({c_2}\oplus_i^{\rightarrow}q_{a_2})\oplus_{i-1}[({c_3}\oplus_i^{\rightarrow}q_{a_3})\oplus_{i-1}[...
\oplus_{i-1}({a_k}\oplus_i^{\rightarrow}q_{a_k})]]],$$

\noindent where the superscript arrow has been reversed and the composition of the $\oplus_{i-1}$ is also reversed. If true, this could be seen as the `dual' of theorem \ref{uniquefactorization}.

\section{Primes, Diophantine and Zeta Connections}

We will see in this section that generalizing primes to arbitrary operations permit us to establish links between primes, diophantine equations and Dirichlet series. In this section we will limit ourselves to binary operations over an arbitrary set $\mathbb{M}$, but in section \ref{generalsection} we will consider $n$-ary functions over arbitrary sets.

\subsection{Generalized Primes and Composites}

According to Hardy and Wright \cite{Hardy}, a positive integer $p$ is said to be prime if $p>1$ and $p$ has no positive divisors except $1$ and $p$. These primes can be viewed as primes associated to the
binary operation of usual multiplication.
In the following we will extend the concept of primes by defining a
prime associated to arbitrary binary operations over the set  $\mathbb{M}$. Note that the notation  $\mathbb{M}$ is used to remind us that $\mathbb{M}$ can be chosen to be a set such as the positive integers $\mathbb{N}=\{1,2,3,...\}$ or the integers $\mathbb{Z}$.

\subsubsection{Preliminaries}
For the arbitrary binary operations we will favor an infix notation to help us keep the intuition of factorization and of primes associated to the usual multiplication. In section \ref{generalsection}, we will see how to define primes associated to arbitrary n-ary operations over arbitrary sets by using a functional notation instead.

\begin{definition}Let
$f:\mathbb{M}\times\mathbb{M}\rightarrow\mathbb{M}$ be a partial function, we
define the \emph{binary operation} $\oplus^{\mathbb{M}}_{f(x,y)}$ over $\mathbb{M}$ as
$$a\oplus^{\mathbb{M}}_{f(x,y)}b=f(a,b).$$ When $f(x,y)$ and $\mathbb{M}$ are not given or are clear by the context, we will write $\oplus_f$ instead of $\oplus^{\mathbb{M}}_{f(x,y)}$\end{definition}

Remark that in this definition of binary operation, $f$ is a partial operation so that operations such as division are considered to be binary operations. This is due to the fact that if $f$ would be restricted to functions, then by taking $\mathbb{M}$ to be the non-negative integers, we would encounter a problem with the division by $0$.

Addition and multiplication are commutative and associative operations, but not all binary operations are commutative and associative. For example:

\medskip

$\oplus^{\mathbb{N}}_{x^2 y^2}$ and  $\oplus^{\mathbb{N}}_{x^2 + y^2}$ are commutative and not associative.

\smallskip

$\oplus^{\mathbb{Z}}_{kxy}$,  $\oplus^{\mathbb{Z}}_{x+xy+y}$ and  $\oplus^{\mathbb{Z}}_{x+y+k}$ are commutative and associative.

\smallskip

$\oplus^{\mathbb{R}}_{y}$ a binary operation is not commutative but is associative, since if $a\neq b$ we have that $$a\oplus^{\mathbb{R}}_y b=b\neq a=a\oplus^{\mathbb{R}}_y b.$$

We now want to extend the concept of factors which is important for the definition of generalized primes. For associative and non-commutative operations it would be enough to define factors in the following way:

\medskip

\noindent``Let $m\in \mathbb{M}$, we say that $d$ is a $\oplus^{\mathbb{M}}_f$-factor of $m$ if and only if there are some $c_1,c_2\in \mathbb{M}$ such that $c_1\oplus^{\mathbb{M}}_f d=m$ or
$d\oplus^{\mathbb{M}}_f c_2=m$ or  $c_1\oplus^{\mathbb{M}}_f d\oplus^{\mathbb{M}}_f  c_2=m$.''

\medskip

But some operations are non-associative and non-commutative, thus we need to first define combinations and occurrences.

\begin{definition}Let
$f:\mathbb{M}\times\mathbb{M}\rightarrow\mathbb{M}$ be a function, we
define the set $C$ of all \emph{$\oplus_{f}$-combinations} and \emph{occurrences} as:
\begin{itemize}
\item[1.] If $x\in\mathbb{M}$ then $x\in C$ and we say that $x$ is an occurrence in $x$,

\item[2.] If $x,y\in C$ then $(x\oplus_{f} y)\in C$ and we say that $x$, $y$ and $(x\oplus_{f} y)$ are occurrences in  $(x\oplus_{f} y)$ and that all occurrences in $x$ and $y$ are also occurrences in $(x\oplus_{f} y)$.

\end{itemize}
\end{definition}

By rule 2 of the previous definition, we make sure that, for example, $3$ is an occurrence in the $\oplus_f$-combination $(4\oplus_f (6 \oplus_f (((12\oplus_f 3)\oplus_f 42)\oplus_f 15)))$.

\begin{definition}We say that $c\in C$ is a
$\oplus_{f}$-representation of $m\in\mathbb{M}$ if
$m=c$.\end{definition}

\begin{definition}Let $m\in \mathbb{M}$, we say that $d$ is a $\oplus^{\mathbb{M}}_f$-factor of $m$ if and only if $d$ is an occurrence in a representation of $m$.\end{definition}

Note that in general there might be many different representations for $m$ where two representations do not have the same factors.

\subsubsection{Units}

We now need to extend the concept of the unit. For the usual primes, $1$ is the unit since for all $k$ we have $1\cdot k=k=k\cdot 1$. In this case, the unit $1$ could be seen as a left-unit and as a right-unit, but we do not have to make this distinction since multiplication is commutative. In the general case, we need to make this distinction.

What makes the unit different from the primes and the composites is the fact that every integer $k$ can be written as a product of that unit and another integer $b$, in particular, in the case of the multiplication, the integer $b$ is $k$ itself. Take the operation $\oplus^{\mathbb{Z}}_{xy-3}$, thus $a\oplus^{\mathbb{Z}}_{xy-3}b=ab-3$ and we find that every integer $m$ can be written as $k=1\oplus^{\mathbb{Z}}_{xy-3}(k+3)$. For example, we have that $1=1\oplus^{\mathbb{Z}}_{xy-3}4$, $2=1\oplus^{\mathbb{Z}}_{xy-3}5$ and $3=1\oplus^{\mathbb{Z}}_{xy-3}6$. If we would define a unit $r$ to be such that $r\oplus s=s$, there are cases, such as seen above, where all integers would be considered to be composites. In particular, although $k=1\oplus^{\mathbb{Z}}_{xy-3}(k+3)$ we have that $1$ would not be considered as a unit because $1\oplus^{\mathbb{Z}}_{xy-3}s=s-3\neq s$. This tells us that it is not the most convenient definition. So we avoid this situation by defining the units as follows.

\begin{definition}We say that
\begin{itemize}
\item[1.] $u\in\mathbb{M}$ is a $\oplus^{\mathbb{M}}_{f}$-left-unit if for each $k\in \mathbb{M}$ there is some $b\in\mathbb{M}$ such that $k=u\oplus^{\mathbb{M}}_{f}b$.

\item[2.]  $u\in\mathbb{M}$ is a $\oplus^{\mathbb{M}}_{f}$-right-unit if for each $k\in \mathbb{M}$ there is some $b\in\mathbb{M}$ such that $k=b\oplus^{\mathbb{M}}_{f}u$.
\end{itemize}

\noindent Sometimes, we will say that an element of $\mathbb{M}$ is a $\oplus^{\mathbb{M}}_f$-unit if it is a $\oplus^{\mathbb{M}}_{f}$-left-unit or $\oplus^{\mathbb{M}}_{f}$-right-unit.

\end{definition}

 It is interesting to note that there are cases where the $b$ is not unique. For example, for the operation $\oplus^{\mathbb{Z}}_{x\mid y\mid}$ where $\mid y\mid$ is the absolute value of $y$, we have that $k=1\oplus^{\mathbb{Z}}_{x\mid y\mid}(-k)=1\oplus^{\mathbb{Z}}_{x\mid y\mid}k$

In the above example for the operation $\oplus^{\mathbb{Z}}_{xy-3}$, we now have that $1$ is considered to be a $\oplus^{\mathbb{Z}}_{xy-3}$-unit since $k=1\oplus^{\mathbb{Z}}_{xy-3}(k+3)$. Here, for each $k$ the $b$ takes the form of the function $k+3$. In some cases it is possible to explicitly find this function, so we will make a small digression before formally defining the generalized composites. We need the following notation.

\begin{definition}\label{curry}Let $f_{x_0}^{-1}:\mathbb{M}\times\mathbb{M}$ denote the inverse of the function $f_{x_0}(y):\mathbb{M}\rightarrow \mathbb{M}$ which is defined as $f_{x_0}(y)=f(x_0,y)$ for a fixed $x_0\in\mathbb{M}$.

Similarly, let $f_{y_0}^{-1}:\mathbb{M}\times\mathbb{M}$ denote the inverse of the function $f_{y_0}(x):\mathbb{M}\rightarrow \mathbb{M}$ which is defined as $f_{y_0}(x)=f(x,y_0)$ for a fixed $y_0\in\mathbb{M}$. \end{definition}

In other words, the function $f_{x_0}(y)$ was obtained by currying the binary function $f(x,y)$ in $x$ or  that the unary function $f_{x_0}(y)$ was obtained by fixing the first argument of the function $f(x,y)$.

In the above example for the operation $\oplus^{\mathbb{Z}}_{xy-3}$, we have that $f(x,y)=xy-3$. Thus, fixing the first argument, we can define the function $f_{x_0}(y)=x_0y-3$. Let's find the inverse of that function. We must have $f_{x_0}(f^{-1}_{x_0}(z))=z$, thus $z=f_{x_0}(f^{-1}_{x_0}(z))=x_0f^{-1}_{x_0}(z)-3.$ By isolating $f^{-1}_{x_0}(z)$ we find that $f^{-1}_{x_0}(z)=\frac{z+3}{x_0}$. If we take $x_0=1$, we have that $f^{-1}_{1}(z)=z+3$ and we observe that $1\oplus^{\mathbb{Z}}_{xy-3}f^{-1}_{1}(k)=1\oplus^{\mathbb{Z}}_{xy-3} (k+3)=k$. From these considerations, we can extract the following proposition.

\begin{proposition}Let the notation be as given in definition \ref{curry}. If $f_{x_0}$ is a bijective function and $u$ is a $\oplus^{\mathbb{M}}_{f(x,y)}$-left-unit, then for each $k\in\mathbb{M}$ there exists a unique $b\in\mathbb{M}$ such that $k=u\oplus^{\mathbb{M}}_{f(x,y)} b$. Furthermore, this $b$ is such that $b=f_u^{-1}(k)$.\end{proposition}
\begin{proof}We must verify that $b=f_u^{-1}(k)$ satisfies $k=u\oplus^{\mathbb{M}}_{f(x,y)} b$. We have that $u\oplus^{\mathbb{M}}_{f(x,y)} f_u^{-1}(k)= f(u,f_u^{-1}(k))$, but by definition \ref{curry}, we can write $f(u,f_u^{-1}(k))$ as $f_u(f_u^{-1}(k))$, hence we find that $u\oplus^{\mathbb{M}}_{f(x,y)} f_u^{-1}(k)=f_u(f_u^{-1}(k))=k$.

Suppose there is some $b'\neq f_u^{-1}(k)$ which satisfies $k=u\oplus^{\mathbb{M}}_{f(x,y)} b'$, thus we have $k=u\oplus^{\mathbb{M}}_{f(x,y)} b'=f(u,b')=f_u(b')$.We know by assumption that $f^{-1}_u(k)=b$, thus $k=f_u(b)$ by applying $f_u$ on both sides. Hence, we have that $f_u(b')=k$ and $f_u(b)=k$, a contradiction with $f_u$ being bijective.\end{proof}

As a side comment, we can also formally extend the concept of zero as follows.

\begin{definition}We say that $z\in\mathbb{M}$ is a
$\oplus_{f}$-left-zero if for all $k\in \mathbb{M}$ we have
$z\oplus_{f}k=z$ and that $z\in\mathbb{M}$ is a
$\oplus_{f}$-right-zero if for all $k\in \mathbb{M}$ we have
$k\oplus_{f}z=z$. \end{definition}

\subsubsection{Generalized Composites}

Let's now find an adequate definition for the generalized composites. A direct generalization of the usual composites definition could be:

\medskip

\noindent\textit{$m\in\mathbb{M}$ is a $\oplus^{\mathbb{M}}_{f}$-composite if
and only if $m=c\oplus^{\mathbb{M}}_{f}d$ and such that  $c$ and $d$ are not $\oplus^{\mathbb{M}}_{f}$-units or $m$ itself.}

\medskip

But, we would encounter at least two problems with this statement. For multiplication, we do not wonder if $1$ is a composite number. For arbitrary operations, a unit could be seen as composite. For example, taking again $\oplus^{\mathbb{Z}}_{xy-3}$, we know that $1$ is a $\oplus^{\mathbb{Z}}_{xy-3}$-left-unit, but also, $1=2\oplus^{\mathbb{Z}}_{xy-3}2$ which would mean that $1$ is a composite. Having a composite unit would make all integers into composites. In the present case, since $16\neq xy-3$ has no integer solutions for $x,y\neq 1$, assuming that $1$ is not a composite would make integers such as $16$ into non-composites (i.e. generalized primes) which gives a richer theory. Thus, we will include in the definition that generalized composites are not units.

Another problem is that there are operations such that for some $m'$ we have $m'=m'\oplus b$. An example is $\oplus^{\mathbb{N}}_{x^2-2y}$ for which $4=4\oplus_{x^2-y}6=16-12$ where $4$ is not a left unit and $6$ is not a right unit.

Thus, to avoid those issues, we define the generalized composites by making sure that they cannot be units and that the number itself can appear as one of its factors.

\begin{definition}$m\in\mathbb{M}$ is a $\oplus^{\mathbb{M}}_{f}$-composite if
and only if it is not a $\oplus^{\mathbb{M}}_{f}$-unit and there exists some $x,y\in\mathbb{M}$ such that $u$ is not a $\oplus^{\mathbb{M}}_{f}$-left-unit and $v$ is not a $\oplus^{\mathbb{M}}_{f}$-right-unit and $(u\oplus_f v)$ is a $\oplus^{\mathbb{M}}_{f}$-factor of $m$.\end{definition}

It is interesting to note that a right-unit could still appear as $u$  where $(u\oplus_f v)$ is a factor of $m$. For example, we have that $12$ is a $\oplus^{\mathbb{N}}_{x-y+8}$-composite since $12=8\oplus^{\mathbb{N}}_{x-y+8} 4$. We have that $4$ is not a right-unit since $3\neq c\oplus^{\mathbb{N}}_{x-y+8} 4$ for any $c\in\mathbb{N}$ and that $8$ is not a left-unit since $17\neq 8\oplus^{\mathbb{N}}_{x-y+8}d$ for any $d\in\mathbb{N}$. But, we have that $8$ is a right-unit since $k=k\oplus^{\mathbb{N}}_{x-y+8}8$.

\subsubsection{Generalized Primes}

We define the generalized primes as:

\begin{definition}$m\in\mathbb{M}$ is a
$\oplus^{\mathbb{M}}_{f}$-prime if and only if it is not a $\oplus^{\mathbb{M}}_{f(x,y)}$-unit and it is not a $\oplus_f$-composite.\end{definition}

By the definition of the generalized units, composites and primes, we have that an element of $\mathbb{M}$ is either a unit, a composite or a prime. In other words, the collection of the sets containing all units, composites and primes is a partition of $\mathbb{M}$.

After some logical manipulations, we find a more useful statement for generalized primes.

\begin{proposition}$m\in\mathbb{M}$ is a $\oplus^{\mathbb{M}}_{f}$-prime if
and only if it is not a $\oplus^{\mathbb{M}}_{f}$-unit and if for all $x,y\in\mathbb{M}$ we have that $x$ is not a $\oplus^{\mathbb{M}}_{f}$-left-unit and $y$ is not a $\oplus^{\mathbb{M}}_{f}$-right-unit implies that  $(x\oplus^{\mathbb{M}}_f y)$ is not a $\oplus_{f}$-factor of $m$.\end{proposition}

\begin{proof}Not being a
$\oplus_f$-composite is the negation of the statement of the definition of $\oplus_f$-composites.

Thus, by De Morgan's Law $\neg(p\wedge q)\Leftrightarrow (\neg p \vee \neg q)$ and the identity $\exists x,y\mathbb{M}P(x,y)\Leftrightarrow\forall x,y\mathbb{M} \neg P(x,y)$, not being a composite is equivalent to:

\medskip

\noindent ``is a $\oplus_{f}$-unit or [for all $x,y\in\mathbb{M}$ we have that $x$ is a $\oplus_{f}$-left-unit or $y$ is a $\oplus_{f}$-right-unit or $(x\oplus_f y)$ is not a factor of $m$].''

\medskip

By the De Morgan's Law and the identity $(\neg u\vee v)\Leftrightarrow (u\Rightarrow v)$ we have that
  $$(p\vee q \vee r)\Leftrightarrow(\neg(p\wedge q) \vee r)\Leftrightarrow ((p\wedge q) \Rightarrow r)$$
Thus, not being a composite is equivalent to:

\medskip

 \noindent ``is a $\oplus_{f}$-unit or [for all $x,y\in\mathbb{M}$ we have that $x$ is not a $\oplus_{f}$-left-unit and $y$ is not a $\oplus_{f}$-right-unit then $(x\oplus_f y)$ is not a factor of $m$].''

\medskip

Therefore since

\begin{center}\begin{tabular}{lll}
$u \wedge (\neg u \vee s)$&$\Leftrightarrow$&$(u\wedge \neg u)\vee (u\wedge s)$\\
&$\Leftrightarrow$&$0\vee (u\wedge s)$\\
&$\Leftrightarrow$&$u\wedge s$\\
\end{tabular}\end{center}
and that being a prime is equivalent by definition to:

\medskip

\noindent ``not a $\oplus_{f}$-unit and it is not a $\oplus_f$-composite,''

\medskip

\noindent we find by replacing it's statement ``not a composite'', that being a prime is indeed equivalent to the statement of the theorem.

\medskip

\end{proof}

As discussed before, the units are not considered to be composites, since this would imply that all elements of $\mathbb{M}$ would be composites if there is a unit in $\mathbb{M}$. Also, the units are not considered to be primes either since we want to be consistent with the definition of the usual primes. Note that
adding the restriction that the unit is not a prime becomes useful
when uniqueness factorization is discussed.

We see that the set of all $\oplus^{\mathbb{N}}_{x+y}$-primes, or $\oplus_{0}$-primes using the hyperoperation notation, is the set $\{1\}$.

The set of $\oplus^{\mathbb{N}}_{xy}$-primes, or $\oplus_{1}$-primes using the hyperoperation notation, are the usual primes and $1$
is a $\oplus^{\mathbb{N}}_{xy}$-left-unit and $\oplus^{\mathbb{N}}_{xy}$-right-unit . By Euclid's proof, we know that the set of primes is infinite and we note that it includes the set of
$\oplus_0$-primes.

The set of $\oplus^{\mathbb{N}}_{x^y}$-primes are all the positive integers that cannot be
written in the form $x^y$ for $y>1$ and $x>1$. This includes all the $\oplus_{ab}$-primes
(the usual primes) but also
includes many $\oplus^{\mathbb{N}}_{ab}$-composites such as $6,10,12,14,...$. Here
is the list of all the $\oplus^{\mathbb{N}}_{x^y}$-primes up to $28$:

$$\{1,2,3,5,6,7,10,11,12,13,14,15,17,18,19,20,21,22,23,24,26,28\}.$$

For hyperoperations, the statement  ``the set of $\oplus_i$-primes contains the set of $\oplus_{i-1}$-primes'' might be an important fact regarding the development of a `fundamental theorem of hyperarithmetic'.

If $f(x,y)=1$, we have that the set of $\oplus^{\mathbb{N}}_{f(x,y)}$-primes are every integer higher than $1$.

Also, if $k>1$ and $f(x,y)=ky$, we have that there is no $\oplus^{\mathbb{Z}}_{ky}$-right-unit, since if we suppose $u$ to be a unit, then all $m$ can be written as $m=c\oplus^{\mathbb{Z}}_{ky}u=ku$, a contradiction since there are integers which are not multiples of $k$. We also have that the set $C^{\mathbb{Z}}_{ky}$ of $\oplus^{\mathbb{Z}}_{ky}$-composites consists of all multiples of $k$ since $kn=b\oplus^{\mathbb{Z}}_{ky}n$ for any $b$ which is not a $\oplus^{\mathbb{Z}}_{ky}$-left-unit and $n\geq 1$. Note that we can write $C^{\mathbb{Z}}_{ky}=k\mathbb{Z}$ which means that $C^{\mathbb{Z}}_{ky}$ can also be seen as an ideal.

\subsection{Diophantine Equations}

We can now associate to each binary operation a set of natural numbers
that exhibit properties similar to the usual primes. In elementary number
theory, there are many theorems and tools that uses the prime concept.
We will now show that it is possible to recast Diophantine problems as problems about primes.

\begin{definition}Let $f:\mathbb{M}\times\mathbb{M}\rightarrow\mathbb{M}$ be a partial function in two variables, then a solution $(x_0, y_0)$ to the equation $f(x,y)=c$ for a fix $c\in \mathbb{M}$ and for $x_0, y_0\in\mathbb{M}$ is said to be \emph{trivial in $x$} if for each $z\in \mathbb{M}$ there is an element $b\in\mathbb{M}$ such that $f(x_0,b)=z$.

Similarly, a solution $(x_0, y_0)$  is said to be \emph{trivial in $y$} if for each $z\in \mathbb{Z}$ there is an element $b\in\mathbb{N}$ such that $f(b,y_0)=z$. Every solution which is not trivial in $x$ and in $y$ is said to be \emph{non-trivial}.\end{definition}

Fix $c$, then an example of a trivial solution in $y$ to the equation $(x^6+y^3)^{\frac{1}{6}}=c$ is the solution $(c,0)$, because for all $z\in\mathbb{Z}$ we can take $b=z$ and find that $f(z,0)=((z)^6+(0)^3)^{\frac{1}{6}}=z$. For the same equation, a trivial solution in $x$ is $(0,c^2)$, since for all $z\in\mathbb{Z}$ we can take $b=z^2$ and find that $f(0,z^2)=((0)^6+(z^2)^3)^{\frac{1}{6}}=z$.

We now have a proposition linking generalized composites to solution of certains diophantine equations.

\begin{theorem}\label{dio}Fix $c\in\mathbb{Z}$ , let $f:\mathbb{Z}\times\mathbb{Z}\rightarrow\mathbb{Z}$ be a polynomial in two variables, then $f(a,b)=c$ has a non-trivial solution $(a_0,b_0)$ for $a_0,b_0\in\mathbb{Z}$ if and only if $c$ is $\oplus_{f}$-composite.\end{theorem}

\begin{proof}($\Rightarrow$) Let $(a_0,b_0)$ be a non-trivial solution. Hence we have that $c=f(a_0,b_0)=a_0\oplus_f b_0$. If $a_0$ is a $\oplus_f$-left-unit, then $a_0\oplus_f b=k$ for all $k$, a contradiction with the non-triviality of the solution. In this manner, we have that $a_0$ is not a $\oplus_f$-left-unit and  $b_0$ is not a $\oplus_f$-right-unit, thus $c$ is a $\oplus_{f}$-composite.

($\Leftarrow$) Conversely, if $c$ is a $\oplus_{f}$-composite, we have
$c=a_0\oplus_{f}b_0=f(a_0,b_0)$ where $a_0$ is not a $\oplus_f$-left-unit and $b_0$ is not a $\oplus_f$-right-unit, thus $(a_0,b_0)$ is not a trivial solution.
 \end{proof}

\begin{theorem}\label{dio-prime-correspondance}Fix $c\in\mathbb{Z}$ , let $f:\mathbb{Z}\times\mathbb{Z}\rightarrow\mathbb{Z}$ be a polynomial in two variables and $h:\mathbb{Z}\rightarrow\mathbb{Z}$ be a polynomial, then $f(a,b)=h(c)$ does not have a non-trivial solution $(a_0,b_0)$ for $a_0,b_0\in\mathbb{Z}$ if and only if $c$ is $\oplus_{f}$-prime.\end{theorem}
\begin{proof}The proof is similar to the proof of theorem \ref{dio}.\end{proof}

The previous theorems are very simple in nature, but they have the
advantage of reformulating some diophantine problems into questions about
generalized composites and primes and vice versa. Since there are many tools concerning the primes we can hope to extend those tools and use them to answer some questions about diophantine
equations. In particular, we could extend certain primality tests which rely on the Fermat's little theorem. In section \ref{fermatlittlesection}, we are taking a few steps in this direction.

\begin{proposition}If $f:\mathbb{Z}\times\mathbb{Z}\rightarrow\mathbb{Z}$ is a polynomial in two variables and $h:\mathbb{Z}\rightarrow\mathbb{Z}$ is a polynomial such that the image of $\mathbb{Z}$ under $h$ is an infinite subset of $\mathbb{Z}$ and if there is a finite number of
$\oplus^{\mathbb{Z}}_{f}$-primes then the diophantine equation $h(z)=f(x,y)$
has a non-trivial solution $(x_0,y_0,z_0)$.\end{proposition}
\begin{proof}Suppose $h(z)=f(x,y)$ has no non-trivial solution for all
$z\in\mathbb{Z}$, then by the previous theorem, each element of the
 image of $\mathbb{Z}$ under $h$ must be a $\oplus_{f}$-prime. Hence a contradiction, since
there is only a finite number of $\oplus^{\mathbb{Z}}_{f}$-primes and  the image of $\mathbb{Z}$ under $h$ is an infinite subset.\end{proof}

We will see in section \ref{multi} that the concepts seen in the present section can be constructed for a more general setting.

\subsection{Uniqueness and the Riemann Zeta Function}

We will define in this section a certain L-series associated to an operation $\oplus^{\mathbb{M}}_f$ and we will see that when this L-series is equal to the Riemann Zeta Function under certain conditions, it means that each element of $\mathbb{M}$ has a unique factorization.

The Euler product formula provides a way to encode primes into a Dirichlet series. If $s>1$ and $p_k$ is the $k^{th}$ common prime (or $\oplus_{xy}$-prime), then
$$\displaystyle\prod_{k=1}^{\infty} \frac{1}{1-p_k^{-s}}=\displaystyle\sum_{n=1}^{\infty} \frac{1}{n^s}=\zeta(s).$$

In general, we would also like to encode generalized primes, but we remark that the Euler product formula relies on the fact that multiplication is associative, commutative and distributive over addition. Looking at the proof of the Euler product formula, there is an intermediary formula which we can use.

 \begin{eqnarray} \displaystyle\prod_{k=1}^{\infty} \frac{1}{1-p_k^{-s}}&=& \displaystyle\prod_{k=1}^{\infty}\left[\displaystyle\sum^{\infty}_{i=1} \left(\frac{1}{p_k^s}\right)^{i}\right]\\
&=&1+\displaystyle\sum_{1\leq i}\frac{1}{p_i^s}+\displaystyle\sum_{1\leq i\leq j}\frac{1}{p_i^sp_j^s}+\displaystyle\sum_{1\leq i\leq j\leq k}\frac{1}{p_i^sp_j^sp_k^s}+...\\
&=& \displaystyle\sum_{n=1}^{\infty} \frac{1}{n^s}
\end{eqnarray}

Here, equation 2 is obtained by the geometric series formula, equation 3 is obtained by the distributive law and 3 is obtained because of the fundamental theorem of arithmetic. Since for an arbitrary operation $\oplus_f$ we do not necessarily have commutativity, associativity and distributivity, it is not clear if a higher geometric series can be found. The main focus here is to provide a way to encode our generalized primes into a Dirichlet series, hence we can more or less start at equation 2. We explore this possibility in the following.

\begin{definition}The set of all $\oplus^{\mathbb{M}}_f$-composites will be denoted by $C^{\mathbb{M}}_f$ and the set of all $\oplus^{\mathbb{M}}_f$-primes will be denoted by $P^{\mathbb{M}}_f$.
\end{definition}

\begin{definition}We define the set of $\oplus^{\mathbb{M}}_{f}$-\emph{prime combinations}
$F_f$ as:
\begin{itemize}
\item[1.] if $u$ is a $\oplus_{f}$-unit then $u\in F_f$;

\item[2.] if $p\in P^{\mathbb{M}}_{f}$ then $p\in F_f$,

\item[3.] if $x,y\in F_f$ and are not units then $(x\oplus_{f}y)\in F_f$.

\end{itemize}
\end{definition}

\begin{definition}\label{lseries}Let $f:\mathbb{M}\rightarrow\mathbb{M}$, $L_{T_f}:\mathds{C}\rightarrow\mathds{C}$, $\mathbb{M}\subseteq\mathbb{C}$ and $T_f\subseteq F_f$, then we define the \emph{$L_{T_f}$-series}
$L_{T_f}$ as $$L_{T_f}(s)=\displaystyle\sum_{t\in T_f} \frac{1}{t^s}$$\end{definition}

Here, we put the restriction $\mathbb{M}\subseteq\mathbb{C}$, since we want each $t\in T_f$ to be in $\mathbb{C}$ so that we have that each component $\frac{1}{t^s}$ of the sum is in $\mathbb{C}$. In the following,  we will mostly restrict ourselves to binary operations with $f:\mathbb{N}\times\mathbb{N}\rightarrow \mathbb{N}$.

For example, take $f:\mathbb{N}\rightarrow\mathbb{N}$ with $f(x,y)=xy$ and take $$T_f=\{3,5,3\oplus_f 5,5\oplus_f 3\},$$ then

 \begin{eqnarray} L_{T_f}(s)&=&  \frac{1}{(3)^s}+ \frac{1}{(5)^s}+ \frac{1}{(3\oplus_f 5)^s}+ \frac{1}{(5\oplus_f 3)^s}\nonumber \\
&=& \frac{1}{(3)^s}+ \frac{1}{(5)^s}+ \frac{1}{(15)^s}+ \frac{1}{(15)^s}\nonumber \\
&=&  \frac{1}{3^s}+ \frac{1}{5^s}+ \frac{2}{15^s}\end{eqnarray}

\begin{proposition}If $f:\mathbb{N}\times\mathbb{N}\rightarrow \mathbb{N}$ and if $c_{i}$ is the number of elements of $T_f$ which are equal to $i$, then $L_{T_f}(s)$ can be rearranged as follow: $$L_{T_f}(s)=\displaystyle\sum_{t\in T_f} \frac{1}{t^s}=\displaystyle\sum_{i=1}^{\infty} \frac{c_{i}}{i^s}.$$\end{proposition}
\begin{proof}The elements of $T$ are combinations of the $\oplus_f$-primes and give after computation a positive integer. For each $i$, we can gather all the terms of the form $\frac{1}{i^s}$ and write $\frac{c_i}{i^s}$ instead. Hence $\displaystyle\sum_{t\in T} \frac{1}{t^s}=\displaystyle\sum_{i=1}^{\infty} \frac{c_{i}}{i^s}$ where $c_{i}\in \mathbb{N}\cup\{0\}$. \end{proof}

In definition \ref{lseries}, take $T=F_{xy}$ such that the set $F_{xy}$ is the set of $\oplus^{\mathbb{N}}_{xy}$-formulas built from the usual multiplication. Thus we have $$L_{F_{xy}}(s)=\displaystyle\sum_{i=1}^{\infty} \frac{c_{i}}{i^s}.$$
Since the $c_i$ count the number of elements of $F_{xy}$ which are equal to $i$ and since for multiplication we have the uniqueness of representation, $c_i$ will count the number of permutations of the primes and the number of associative ways to write the brackets. Not counting these provides a way to directly connect this L-series to the Zeta function which will be the purpose of the following statements.

We now define the set of formulas modulo associativity and commutativity. Note that we could also define this by using quotient sets, but there is an advantage to this approach, since it invites the construction of subsets that are not necessarily built by using relations.

\begin{definition}Let $F_f$ be the set of $\oplus^{\mathbb{M}}_{f}$-\emph{prime combinations}, we define the set of \emph{$\oplus^{\mathbb{M}}_{f}$-prime combinations modulo associativity and commutativity}
$T_{AC}$ as:
\begin{itemize}
\item[1.] if $u$ is a $\oplus^{\mathbb{M}}_{f}$-unit then $u\in T_{AC}$,

\item[2.] if $p\in P^{\mathbb{M}}_{\oplus_{f}}$ then $p\in T_{AC}$,

\item[3.] if $x,y\in F_f$ and are not units and $(x\oplus_{f}y)\notin T_{AC}$ then $(y\oplus_{f}x)\in T_{AC}$,

\item[4.] if $x,y,z\in F_f$ and are not units and $(x\oplus_{f}(y\oplus_{f}z))\notin T_{AC}$ \\ then  $((x\oplus_{f}y)\oplus_{f}z)\in T_{AC}$.
\end{itemize}
\end{definition}

\begin{proposition}If the operation $\oplus^{\mathbb{N}}_f$ is commutative and associative, then for each positive integer $n$ there is at least an element $t\in T_{AC}$ such that $n=t$.\end{proposition}
\begin{proof}By definition, all the $\oplus^{\mathbb{N}}_f$-primes and the $\oplus^{\mathbb{N}}_f$-units are in $T_{AC}$. Suppose there is a composite positive integer $c$ not in $T_{AC}$, then by rules 3 and 4 of the definition of $T_{AC}$ this means that if the composite is of the form $(y\oplus_{f}x)$ it is not in $T_{AC}$ because $(x\oplus_{f}y)$ is in $T_{AC}$ or  $((x\oplus_{f}y)\oplus_{f}z)$ is not in $T_{AC}$ because  $((x\oplus_{f}y)\oplus_{f}z)$ is in $T_{AC}$. Therefore we must have that
 $$c=(y\oplus_{f}x)\neq (x\oplus_{f}y)$$
 or
$$c=((x\oplus_{f}y)\oplus_{f}z)\neq ((x\oplus_{f}y)\oplus_{f}z)$$
But this is a contradiction with the commutativity or associativity of $\oplus^{\mathbb{N}}_f$.

\end{proof}

Now that we know each positive integer has a $\oplus_f$-representation in $T_{AC}$, we can state the following theorem.

\begin{theorem}Let $\oplus^{\mathbb{N}}_f$ be associative and commutative, then for all real numbers $s>1$ we have $L_{T_{AC}}(s)=\zeta(s)$ if
and only if for each $n\in\mathbb{N}$  there is exactly one element $c\in T_{AC}$ such that $n=c$ or, in other words, each positive integer has a unique $\oplus_f$-representation up to associativity and commutativity. \end{theorem}
\begin{proof}($\Leftarrow$) We need to calculate $L_{T_{AC}}(s)$ and we know that $$L_{T_
{AC}f}(s)=\displaystyle\sum_{t\in T_{AC}} \frac{1}{t^s}=\displaystyle\sum_{i=1}^{\infty} \frac{c_{i}}{i^s}.$$ Since, by definition of the operation, we have that the codomain of $f$ is $\mathbb{N}$, thus we have that each $t$ equals a certain positive integer. Suppose that there is a $c_j>1$, then this means that there are $t_1,t_2\in T_{AC}$ such that $t_1=t_2=n$. But this is a contradiction since, by assumption, we have that each positive integer $k$ has exactly one element of $T_{AC}$ which is equal to $k$. Hence, for all $i$ we have $c_i=1$, and therefore, we find that $L_{T_{AC}}(s)=\displaystyle\sum_{i=1}^{\infty} \frac{1}{i^s}=\zeta(s)$.

($\Rightarrow$) Suppose there is a finite or infinite number of positive integers which are not uniquely represented up to associativity and commutativity, say $\{k_1,k_2,...\}$. Since $\oplus_f$ is associative and commutative, by the previous proposition, each of those integers are in $T_{AC}$ and are not uniquely represented. By the assumption $L_{T_{AC}}(s)=\zeta(s)$, we have that $$0=L_{T_{AC}}(s)-\zeta(s)=\displaystyle\sum_{i=1}^{\infty} \frac{c_{i}-1}{i^s}.$$ Hence, some of the components for which the $c_i$'s are equal to $1$ will cancel out and the components of integers not uniquely represented will remain. Let  $a_i=c_i-1\geq 0$, and define $R(s)$ as follow,  $$R(s)=\displaystyle\sum_{i=1}^{\infty} \frac{a_i}{i^s}$$

By assumption of the theorem, we must have that $R(s)=0$ for all real $s>1$, thus take $s_0>1$ a solution such that $R(s_0)=0$. We must also have that $s_0+1$ is such that $R(s_0+1)=0$, thus $$\displaystyle\sum_{i=1}^{\infty} \frac{a_i}{i^{s_0}i}=0.$$ Remark that each $\frac{a_i}{i^{s_0}}$ are non-negative numbers and that since $i$ is also a positive integer, we have that $$\frac{a_i}{i^{s_0}}>\frac{a_i}{i^{s_0}i}.$$ Hence we have $0=R(s_0+1)>R(s_0)=0$, which is a contradiction. Therefore each positive integer must have a unique $\oplus_f$-representation.  \end{proof}

One of the requirements of the theorem is that the operation is associative and commutative. Some examples of associative and commutative operations are $\oplus_{x+y+k}$ and $\oplus_{kxy}$. From these we can build other commutative and associative operations such as $\oplus_f$ and $\oplus_g$ where $f(u,v)=h\oplus_{kxy}u\oplus_{kxy}v$ and $g(u,v)=h\oplus_{x+y+k}u\oplus_{x+y+k}v$. The function $R(s)$ with $s>1$ can be seen as measuring how far a commutative and associative operation $f$ is from inducing unique factorization.

\subsection{Generalized Modulo Arithmetic}\label{fermatlittlesection}

The initial idea of investigating Modular Arithmetic in the setting of generalized primes was to find a general version of Fermat's little theorem so that we could build a factorization algorithm which would provide a useful way to solve Diophantine problems. This objective is not met in this article, but we hope that this working basis can be extended further in the future.

We will now construct a modulo arithmetic associated to an arbitrary operator. Recall that when we write $\oplus_f$ we mean $\oplus^{\mathbb{M}}_f$.

\begin{definition}We will say that $\oplus_{f_R^{-1}}$ is a \emph{right inverse binary operation} of the binary operation  $\oplus_f$ if $(a\oplus_f b)\oplus_{f_R^{-1}}b=a$. We will say that $\oplus_{f_L^{-1}}$ is a \emph{left inverse binary operation} of the binary operation  $\oplus_f$ if $a\oplus_{f_L^{-1}}(a\oplus_f b)=b$.\end{definition}

For example, subtraction is the right inverse of addition and division is a right inverse of multiplication. Furthermore, using our hyperoperation notation $\oplus_2$ for exponentiation, the logarithm is the right inverse of exponentiation since $$(a\oplus_2 b)\oplus_{\log_{y}x}b=(b^a)\oplus_{\log_{y}x}b=\log_{b}b^a=a.$$  For subtraction, $\oplus_{f_R^{-1}}$ is denoted by the symbol `$-$'. We remark that we can write the left inverse of addition by using the right inverse, since $ a\oplus_{y-x}(a+b)=b$.

\begin{definition}Let $\oplus_f$ be a binary operation and  $\oplus_g$ be a binary operation which has a right inverse binary operation $\oplus_{g^{-1}}$. We write $c\equiv_{g} b\,\,(\mathrm{mod}_{f}\,\,m)$ if and only if $m$ is a $\oplus_{f}$-factor of $c\oplus_{g^{-1}}b$.\end{definition}

A trivial example is when $g$ is addition, $f$ is the usual multiplication (i.e. f(x,y)=xy) and $g^{-1}$ is subtraction. In the case where $g$ is the usual addition we will sometimes write `$\equiv$' instead of `$\equiv_{x+y}$' and in the case where $f$ is the usual multiplication we will sometimes write  `$\mathrm{mod}$' instead of `$\mathrm{mod_{xy}}$'.

Finding a Fermat's little theorem associated to an arbitrary operation $\oplus_f$ can be quite challenging since useful properties, such as addition and multiplication respecting modulo relations and the cancellation law, will not occur in the case of some operators. If $c\equiv_{g} b\,\,(\mathrm{mod}_{f}\,\,m)$ and $c'\equiv_{g} b'\,\,(\mathrm{mod}_{f}\,\,m)$ then what are the conditions on the $f$ and $g$ such that we find $c\oplus_{g}c'\equiv_{g} b\oplus_{g}b'\,\,(\mathrm{mod}_{f}\,\,m)$ or $c\oplus_{f}c'\equiv_{g}b\oplus_{f}b'\,\,(\mathrm{mod}_{f}\,\,m)$?

We will now give a simple extension of Fermat's little theorem.

 \begin{definition}We will write $a^{\oplus^{\leftarrow}_f p}$ instead of $(a \oplus_f (a\oplus_f (...\oplus_f(a\oplus_f a)...)))$ in which $a$ occurs exactly $p$ times.\end{definition}

\begin{theorem}Let $\mathbb{M}=\mathbb{Z}$. If $p$ is a $\oplus_{xy}$-prime (i.e.,
 usual prime) and $\gcd(k,p)=1$, then $$a^{\oplus^{\leftarrow}_{kxy} p}\oplus_{kxy} 1\equiv a\oplus_{kxy}1 \,\, (\mathrm{mod}_{\oplus_{kxy}}\,\, p).$$\end{theorem}
\begin{proof}

By induction, we can show that $a^{\oplus^{\leftarrow}_{kxy} p}=k^{p-1}a^{p}$. Thus, we have that $$a^{\oplus^{\leftarrow}_{kxy} p}\oplus_{kxy} 1=k^{p-1}a^{p}\oplus_{kxy} 1=kk^{p-1}a^{p}.$$ Since $p$ is a usual prime, by Fermat's little theorem, we have $k^{p-1}\equiv 1 \,\,(\mathrm{mod}\,\, p)$ and
$a^{p}\equiv a \,\,(\mathrm{mod}\,\, p)$. Hence, by the modulo multiplication property, we find $k^{p-1}\cdot a^{p}\equiv a\cdot 1\,\,(\mathrm{mod}\,\, p)$. We can multiply each side by $k$, thus
$kk^{p-1}a^{p}\equiv ka\,\,(\mathrm{mod}\,\, p)$. Since $\gcd(k,p)=1$, $p$ is prime and $kk^{p-1}a^{p}-ka\equiv 0\,\,(\mathrm{mod}\,\, p)$, we can write $kk^{p-1}a^{p}-ka=k[k^{p-1}a^{p}-a]=k\alpha p$ for some $\alpha\in\mathbb{Z}$. Since
$k\alpha p=\alpha\oplus_{kxy} p$ then we have that $p$ is a $\oplus_{kxy}$-factor of $kk^{p-1}a^{p}-ka$, hence we have that  $k^{p-1}a^{p}\equiv ka \,\, (\mathrm{mod}_{\oplus_{kxy}}p)$. Since $kk^{p-1}a^{p}=a^{\oplus^{\leftarrow}_{kxy} p}\oplus_{kxy} 1$ and $ka=a\oplus_{kxy} 1$ we find that $a^{\oplus^{\leftarrow}_{kxy} p}\oplus_{kxy} 1\equiv a\oplus_{kxy} 1 (\mathrm{mod}_{\oplus_{kxy}}p)$.
\end{proof}

We remark that by taking $k=1$, the above theorem is reduced to the statement of Fermat's little theorem.

In the following theorem we consider a statement where the index of `$\mathrm{mod}$' is the usual multiplication, but where the modulo equation resemble Fermat's little theorem.

\begin{theorem}Let $u,v\in\mathbb{Z}$ be such that $\gcd(p,v)=1$ and $u=h(v-1)$ for $h\in\mathbb{Z}$, if $p$ is a usual prime, then for all $k\in\mathbb{Z}$ we have that $k^{\oplus^{\leftarrow}_{ux+vy} p}\equiv k \,\,(\mathrm{mod}\,\, p)$.\end{theorem}
\begin{proof}By induction we can prove that $$k^{\oplus^{\leftarrow}_{ux+vy} p}= uk+vuk+v^2uk+...+v^{p-2}uk+v^{p-1}k.$$ We can rewrite the right side as $$k(u(1+v+v^2+...+v^{p-2})+v^{p-1})$$ and by using the geometric
series formula we find $$k^{\oplus^{\leftarrow}_{ux+vy} p}=\displaystyle k\left(\frac{u(v^{p-1}-1)}{(v-1)}+v^{p-1}\right).$$ Hence, by subtracting by $k$ on each side, we have $$k^{\oplus^{\leftarrow}_{ux+vy} p}-k=\displaystyle k\left(\frac{u(v^{p-1}-1)}{(v-1)}+v^{p-1}-1\right).$$
Since $u=h(v-1)$ we find $$k^{\oplus^{\leftarrow}_{ux+vy} p}-k=k(h(v^{p-1}-1)+v^{p-1}-1)=k(v^{p-1}-1)(h+1).$$ Since $\gcd(p,v)=1$ we have by the Fermat's little theorem that $v^{p-1}-1$ is divisible by $p$. Hence we conclude that $k^{\oplus^{\leftarrow}_{ux+vy} p}\equiv k \,\,(\mathrm{mod}\,\, p)$. \end{proof}

Euclid's lemma (i.e.: If $p$ is prime and $p$ divides $ab$, then $p$ divides $a$ or $p$ divides b.) is classically used to prove the cancelation law and the uniqueness of factorization. Some interesting investigations would be to find the connections between a general Euclid's lemma and the uniqueness of factorization.

\section{Language of the Generalized Primes}\label{generalsection}

In this section we will see that the generalized primes can act as an interesting and different language for addressing Diophantine problems, Goldbach type problems and algebraic fields. We will see that this might also lead to the interpretation of certain problems in the general perspective of category theory.

\subsection{Primes Associated to Multivariate Operations}\label{multi}

\begin{definition}We will say that $f$ is an \emph{n-ary operation over $\mathbb{M}$} \\if $f:\mathbb{M}^n\rightarrow\mathbb{M}$ is a n-ary partial function. \end{definition}

\begin{definition}Let
$f$ be a n-ary operation, we
define the set $C$ of all \emph{$f$-combinations} and \emph{occurrences} as:
\begin{itemize}
\item[1.] If $x\in\mathbb{M}$ then $x\in C$ and we say that $x$ is an occurrence in $x$,

\item[2.] If $x_i\in C$ then $f(x_1,x_2,...,x_n)\in C$ and we say that $x_1,x_2,...,x_n$ and $f(x_1,x_2,...,x_n)$ are occurrences in  $f(x_1,x_2,...,x_n)$ and for all $i$ the occurrences in $x_i$ are also occurrences in $f(x_1,x_2,...,x_n)$.
\end{itemize}
\end{definition}

\begin{definition}We say that $c\in C$ is an
$f$-representation of $m\in\mathbb{M}$ if
$m=c$.\end{definition}

\begin{definition}Let $m\in \mathbb{M}$, we say that $d$ is a $f$-factor of $m$ if and only if $d$ is an occurrence in a representation of $m$.\end{definition}

\begin{definition}We say that
 $u\in\mathbb{M}$ is a $[f,j]$-unit if for all $s\in \mathbb{M}$ there are some $b_j\in\mathbb{M}$ such that
 $$s=f(b_{1},b_{2},...,b_{j-1},u,b_{j+1},...,b_{n}).$$
\noindent Sometimes, we will say that it is an $f$-unit if it is an  $[f,j]$-unit for any $j$.
\end{definition}

\begin{definition}$m\in\mathbb{M}$ is an $f$-composite if
and only if it is not a $[f,j]$-unit for all $j$ and there exists some $c_k\in\mathbb{M}$ such that each $c_k$ is not a $[f,j]$-unit and $f(c_1,c_2,...,c_n)$ is an $f$-factor of $m$.\end{definition}

\begin{definition}$m\in\mathbb{M}$ is an
$f$-prime if and only if for all $j$ it is not an $[f,j]$-unit and it is not an $f$-composite.\end{definition}

We now want to generalize proposition \ref{dio-prime-correspondance} where we established a correspondence between solutions of Diophantine equations and composites. As before, we need to define what is a trivial solution.

\begin{definition}Let $f: \mathbb{M}^n\rightarrow \mathbb{M}$ and
$g: \mathbb{M}^m\rightarrow \mathbb{M}$.

\begin{itemize}
\item[1.]We say that
$$(x'_1,x'_2,...,x'_n, y'_1,y'_2,...,y'_m)$$ is a solution over $\mathbb{M}$ of the equation $f(x_1,x_2,...,x_n)=g(y_1,y_2,...,y_m)$ if  $$f(x'_1,x'_2,...,x'_n)=g(y'_1,y'_2,...,y'_m),$$ with $x'_i,y'_i\in\mathbb{M}$ for all $i$.

\item[2.] We say that the solution is trivial in $x_i$ if for all $c\in\mathbb{M}$ there are some indexed $b$ in $\mathbb{M}$ such that $$c=f(b_1,b_{2},...,b_{i-1},x'_i,b_{i+1},...,b_{n}).$$

    Similarly, We say that the solution is trivial in $y_i$ if for all $c\in\mathbb{M}$ there are some indexed $b$ in $\mathbb{M}$ such that $$c=g(b_1,b_{2},...,b_{i-1},y'_i,b_{i+1},...,b_{m}).$$

    Every solution which is nowhere trivial is said to be \emph{non-trivial}.
\end{itemize}
\end{definition}

We now have a generalization of theorem \ref{dio}.

\begin{theorem}\label{diogeneral}Let $f: \mathbb{M}^m\rightarrow \mathbb{M}$ and
$g: \mathbb{M}^m\rightarrow \mathbb{M}$. Let $C_{f}$ and $C_{g}$ be the sets of $f$-composite and $g$-composite respectively. Then, $C_{f}\cap C_{g}\neq \varnothing$ if and only if the equation $f(x_1,x_2,...,x_m)=g(y_1,y_2,...,y_n)$ has a non-trivial solution over $\mathbb{M}$.
\end{theorem}

\begin{proof}($\Rightarrow$)Take $c\in C_{f}\cap C_{g}$, then, since $c$ is $\oplus_f$-composites and $\oplus_g$-composites, we have $$f(a_1,a_2,...,a_m)=c=g(b_1,b_2,...,b_n)$$ where each $a_i\in\mathbb{M}$ is not a $[f,i]$-unit and each $b_i\in\mathbb{M}$ is not a $[g,i]$-unit  which means that  $$(a_1,a_2,...,a_m,b_1,b_2,...,b_n)$$ is a solution which is nowhere trivial.

($\Leftarrow$) Conversely, if $$f(a_1,a_2,...,a_m)=g(b_1,b_2,...,b_n)$$ such that the solution over $\mathbb{M}$ is nowhere trivial hence we have that $f(a_1,a_2,...,a_m)$ is $f$-composite and $g$-composite, thus we have that  $f(a_1,a_2,...,a_m)\in C_{f}\cap C_{g}$ and therefore, $C_{f}\cap C_{g}\neq \varnothing$.
\end{proof}

Another way to look at this is that an equation does not have a non-trivial solution if the sets of primes satisfies a related condition.

\begin{theorem}\label{diogeneral2}Let $f: \mathbb{M}^m\rightarrow \mathbb{M}$ and
$g: \mathbb{M}^m\rightarrow \mathbb{M}$. Let $P_{f},P_{g},U_{f}$ and $U_{g}$ be respectively the sets of $f$-primes, $g$-primes, $f$-units and $g$-units. Then, \\ $P_{f}\cup P_{g}\cup U_{f}\cup U_{g}= \mathbb{M}$ if and only if the equation $f(x_1,x_2,...,x_m)=g(y_1,y_2,...,y_n)$ has no non-trivial solution over $\mathbb{M}$.
\end{theorem}
\begin{proof}($\Rightarrow$) Suppose the equation has a non-trivial solution over $\mathbb{M}$, then by theorem \ref{diogeneral}, there is a composite number $c\in C_{f}\cap C_{g}$. This means that $c$ cannot be an  $f$-prime or a $g$-prime, hence we find that $P_{f}\cup P_{g}\cup U_{f}\cup U_{g}\neq \mathbb{M}$, a contradiction.

($\Leftarrow$) Conversely, suppose there exists some $c\in\mathbb{M}$ such that $c\notin P_{f}\cup P_{g}\cup U_{f}\cup U_{g}$, then this $c$ must be an $f$-composite and a $g$-composite. By the previous theorem the equation has a non-trivial solution over $\mathbb{M}$, a contradiction. \end{proof}

In general, for we will denote the set of primes, units and composites as follows.
\begin{definition}For $f: \mathbb{M}^m\rightarrow \mathbb{M}$, we denote the set of all $f$-primes by $P^{\mathbb{M}}_f$, the set of  $f$-composites by $C^{\mathbb{M}}_f$ and the set of all $f$-units by $U^{\mathbb{M}}_f$ \end{definition}

Using theorems \ref{diogeneral}, we can rephrase Fermat's last theorem as :
$$\mathrm{If}\,\,\,\, n>2,\,\,\,\, \mathrm{then} \,\,\,\,C^{\mathbb{Z}}_{x^n+y^n}\cap C^{\mathbb{Z}}_{z^n}=\varnothing$$

Since, for $n>2$ there are no trivial solutions of $x^n+y^n$ and of $z^n$ in the sense that we defined, thus we have that $U^{\mathbb{Z}}_{x^n+y^n}=U^{\mathbb{Z}}_{z^n}=\varnothing$. Thus, by theorem \ref{diogeneral2}, we can also rephrase Fermat's last theorem as :
$$\mathrm{If}\,\,\,\, n>2,\,\,\,\, \mathrm{then} \,\,\,\,P^{\mathbb{Z}}_{x^n+y^n}\cup P^{\mathbb{Z}}_{z^n}=\mathbb{Z}$$

To be more precise, we would need to indicate that the function $x^n+y^n$ is a binary function and that $z^n$ is a unary function. We can do this by writing $C^{\mathbb{Z}}_{[2,x_1^n+x_2^n]}$ instead of $C^{\mathbb{Z}}_{x^n+y^n}$ and by writing  instead of $C^{\mathbb{Z}}_{[1,y_1^n]}$. In general, we will not write the arity of the function when it is the same as the number of variables or when it is clear by context.

Now, Fermat's last theorem is stated as a problem about generalized primes. It is interesting to see that Diophantine problems can be written as problems about generalized primes and composites. A possible outcome of this is to solve Diophantine equations by using tools of the theory of prime numbers such as extensions of Fermat's little theorem or the tools surrounding Dirichlet series.

\subsection{Goldbach Problems}

Let $f$ be a unary function over the positive integers such that $f(x)=2n$ with $n$ a fixed positive integer, then we have that $C^{\mathbb{N}}_{2n}=\{2n\}$ since the only positive integer which can be written as $2n$ is $2n$ itself. Using the same notation as in the previous section, we can write the Goldbach conjecture as:
$$\mathrm{If}\,\,\,\, n>2,\,\,\,\, \mathrm{then} \,\,\,\,C^{P^{\mathbb{N}}_{xy}}_{x_1+x_2}\cap C^{\mathbb{N}}_{2n}\neq\varnothing$$

Let $\mathbb{N}_0$ be the set of non-negative numbers, then another problem which can be written as above is the Lagrange four square theorem:

$$\mathrm{If}\,\,\,\, n\geq 1,\,\,\,\, \mathrm{then} \,\,\,\,C^{\mathbb{N}_0}_{x_1^2+x_2^2+x_3^2+x_4^2}\cap C_{n}\neq\varnothing.$$

By extending the concept of primes, we now have many new questions arising. Some of those questions take the format of the Goldbach conjecture question. For example, the $\oplus^{\mathbb{N}}_{x^2+y^2}$-primes are
$$\{1,3,4,6,7,9,11,12,14,15,19,21,22,23,24,27,...\}.$$ By inspection, we can conjecture the following:

\begin{conjecture}Each integer higher or equal to $4$ can be written
as the sum of two $\oplus_{x^2+y^2}$-primes.\end{conjecture}

In our notation this can be written as:
$$\mathrm{If}\,\,\,\, n>1,\,\,\,\, \mathrm{then} \,\,\,\,C^{P_{z_1^2+z_2^2}}_{x_1+x_2}\cap C^{\mathbb{N}}_{n}\neq\varnothing.$$

We see that the introduction of the generalized primes gives us a whole new range of questions to consider.

\subsection{Algebraically Closed Fields}

As seen before, it is possible to consider primes associated to a unary operation. For example, define the unary operation $\oplus^{\mathbb{N}}_{x^2}$ as $\oplus^{\mathbb{N}}_{x^2}a=a^2$. Thus, for any integer of the form $b^2$ we have $b^2=\oplus^{\mathbb{N}}_{x^2}b$ hence any square is $\oplus_{x^2}$-composite. Note that for unary operations over $\mathbb{N}$ there are no units.

If we would be over the real numbers instead of the positive integers, we would have that every non-negative real is a $\oplus^{\mathbb{R}}_{x^2}$-composite. In this case, the $\oplus^{\mathbb{R}}_{x^2}$-primes would be all negative reals. Over the complex numbers, we have that all complex numbers are $\oplus^{\mathbb{C}}_{x^2}$-composite. This point of view provides another way to approach different types of numbers. In particular, the integers can be viewed as the $\oplus^{\mathbb{Q}}_{\frac{x}{y}}$-primes, that is, all the rational numbers which cannot be written as a fraction with $y\neq 1$. Note that the $1$ is a $\oplus^{\mathbb{Q}}_{\frac{x}{y}}$-right-unit.

It also provides another way to define algebraically closed fields.
\begin{proposition}A field $F$ is algebraically closed if and only if for every polynomial $p(x)$ over $F$ we have that $0$ is a $\oplus_{p(x)}$-composite.\end{proposition}
\begin{proof} If $F$ is algebraically closed then for each $p(x)$ there is a root $r\in F$ such that $p(r)=0$, but
$0=p(r)=\oplus_{p(x)}r$. Thus $0$ is $\oplus_{p(x)}$-composite since for unary operations there are no units.

Conversely, if for every $p(x)$ we have that $0$ is $\oplus_{p(x)}$-composite, then for some $v\in F$ we have $0=\oplus_{p(x)}v=p(v)$, hence $F$ is algebraically closed.\end{proof}

\subsection{Diophantine Category}
Roughly, elementary number theory can be viewed as the study of integers and formulas which use the usual arithmetic operations. From this comes the concept of the usual primes which give rise to useful tools and deep problems. Since the primes are so fundamental and that we now have a generalized concept of primes it seems worthwhile to consider the theory that would arise from a different set of operations over the integers.

It would be quite some work to formally define what is an elementary number theory associated to a set of operations and it would result in a restrictive concept. The next best thing seems be to the following.

\begin{definition}We will call the \emph{core} of a number theory a finite or infinite collection of operations where each operation is of an arbitrary arity and is over an arbitrary set.\end{definition}

An example of a core is the set of all hyperoperations $\oplus_j$ for $j\in \mathbb{N}$. Here each higher operation depend on the previous, but we do not ask for this condition in every core. A reason is that if we want to study a certain equation consisting of unrelated operations we can still define a core out of those operations and investigate the number theory surrounding that core which could give us interesting insights into the equation.

As seen previously, it is possible to associate to each operation $\oplus_f$ a single set of $\oplus_f$-primes and of $\oplus_f$-composites. Conversely, we could choose a subset of the positive integers and find an operator associated to it, but it need not be unique. For example, take the operator $\oplus^{\mathbb{M}}_f$ defined as $2\oplus^{\mathbb{M}}_f d=3$ for all $d\in\mathbb{M}$ and $b\oplus^{\mathbb{M}}_f c=2$ for all $b,c\in\mathbb{M}$ such that $b\neq 2$, then only $2$ and $3$ are composites. If we define the operator $\oplus^{\mathbb{M}}_g$ as $3\oplus^{\mathbb{M}}_g d=2$ for all $d\in\mathbb{M}$ and $b\oplus^{\mathbb{M}}_g c=3$ for all $b,c\in\mathbb{M}$ such that $b\neq 3$, then also only $2$ and $3$ are composites. This means that we can define two different operators for the set of composites $\{2,3\}$. It would be interesting to find if for some subset of the positive integers there is a unique operation (of a certain type) associated to the subset.

In the following, we briefly consider the context of category theory. For details regarding category theory, \cite{MacLane} is the definitive introduction. If we choose a core, to each operation $\oplus_k$ of the core we can associate the set $C_k$ of all $\oplus_k$-composites. Collecting all those sets and ordering them by inclusions, we build a \textit{composites core category}, which is actually a poset category, where the objects are the sets of $\oplus_k$-composites and the arrows are the set inclusions. In a similar manner, we define the \textit{primes core category} and the \textit{units core category}.

We could now study Diophantine equations by using categorical tools.  For example, in the previous section, we saw that the intersection of composite sets $C_{f}\cap C_{g}$ gives us information about the solutions of the associated diophantine equation. From a categorical point of view this intersection is equivalent to the pullback $C_g\times_{C_g\cup C_f} C_f$. Furthermore, we now have the possibility to investigate the functors between core categories which could be built from different classes of operations such as linear or quadratic operations.

\begin{definition}Let $D$ be a core composed of all polynomials over $\mathbb{Z}$, we will call its composite core category the \emph{composites diophantine category}. In a similar manner, we define the \emph{primes diophantine category} and the \emph{units diophantine category}. \end{definition}

In any Diophantine category, it is interesting to notice that $\mathbb{Z}$ is considered to be a terminal object because there is a unique arrow, or set inclusion, between any subset of $\mathbb{Z}$ and $\mathbb{Z}$. Note that in the composites diophantine category, for each fixed $h\in\mathbb{Z}$, we have that $C^{\mathbb{Z}}_{x+h}$ is a terminal object since $C^{\mathbb{Z}}_{x+h}=\mathbb{Z}$.  This is because any integer $n$ can be written as $n=\oplus^{\mathbb{Z}}_{x+h}n-h$. Similarly, a terminal object in the units diophantine category is given by $U^{\mathbb{Z}}_{x+ky}=\mathbb{Z}$ for each $k\in\mathbb{Z}$.

\section{Conclusion}

Primes have often been seen as mysterious entities and have fascinated many. The generalization of primes has been approached through algebra and algebraic geometry. As seen above, another way to generalize the concept of primes is through the operations themselves. Perhaps that this point of view can eventually be used to include the algebraic view or vice versa.

Regarding the topics presented here, we see that there seem to be many interesting discoveries to be uncovered. One objective is to find a fully generalized `fundamental theorem of hyperarithmetic'. We have seen that the `hyperarithmetics' presented in \ref{hyperarithmetic} can be seen as a core. A step further would be to find the fundamental theorem associated to an arbitrary core. A good understanding of algebraic relations will be needed since generalized distributivity and exponential laws will probably be needed for the proof of such a general theorem.

 A second objective is to clarify the links between Diophantine equations, generalized primes and the theory of Dirichlet series. It would be quite remarkable to move freely between each of these domains, where a problem in one of these domains can be solved in the language of the other two. It is still unclear what will be the input of category theory, but one can hope that its general perspective will help answer deep questions in number theory.

\end{document}